\documentclass[10pt]{amsart}

\usepackage[latin1]{inputenc}
\usepackage[english]{babel}
\usepackage{amsmath}
\usepackage{amsfonts}
\usepackage{amssymb}
\usepackage{amsthm}
\usepackage{epsfig}
\usepackage{graphicx,colortbl}
\usepackage{url,hyperref}
\usepackage{mathtools, color, changes}

\newtheorem{theorem}{Theorem}[section]
\newtheorem{lemma}{Lemma}[section]

\newtheorem{rmk}{Remark}[section]

\newcounter{theor}

\def\R{\mathbb{R}}
\def\N{\mathbb{N}}
\def\Z{\mathbb{Z}}

\def\2Z{2^{-m}\Z^n}

\numberwithin{equation}{section}

\begin{document}

\title[Reduction of the slicing problem to symmetric bodies]{A note on the reduction of the slicing problem to centrally symmetric convex bodies}
\author{Javier Mart\'i­n-Go\~ni}
\address{Departamento de Matem\'aticas. Facultad de Ciencias. Universidad de Zaragoza. Pedro Cerbuna 12, 50009, Zaragoza (Spain), IUMA.}
\email{j.martin@unizar.es, 745533@unizar.es}
\thanks{Partially supported by MICINN project PID-105979-GB-I00 and DGA project E48\_20R.}
\maketitle

\begin{abstract}
In this paper, we obtain the best possible value of the absolute constant $C$ such that for every isotropic convex body $K\subseteq\R^n$ the following inequality (which was proved by Klartag and reduces the hyperplane conjecture to centrally symmetric convex bodies) is satisfied:
$$
L_K\leq CL_{K_{n+2}(g_K)}.
$$
Here $L_K$ denotes the isotropic constant of $K$, $g_K$ its covariogram function, which is log-concave, and, for any log-concave function $g$, $K_{n+2}(g)$ is a convex body associated to the log-concave function $g$, which belongs to a uniparametric family introduced by Ball. In order to obtain this inequality, sharp inclusion results between the convex bodies in this family are obtained whenever $g$ satisfies a better type of concavity than the log-concavity, as $g_K$ is, indeed $\frac{1}{n}$-concave.
\end{abstract}

\section{Introduction and notation}
A convex body $K\subseteq\R^n$ (i.e., a compact convex set with non-empty interior) is called isotropic if it has volume (Lebesgue measure), which we will denote by $|K|$, equal to 1, its barycenter is at the origin, and there exists $L_K>0$, known as the isotropic constant of $K$, such that for every $\theta\in S^{n-1}$, the Euclidean unit sphere,
$$
\int_K\langle x,\theta\rangle^2dx=L_K^2,
$$
where $\langle\cdot,\cdot\rangle$ denotes the standard scalar product of two vectors.

It is well known that for every convex body $K\subseteq\R^n$, there exists a unique (up to orthogonal maps) affine map $a+T$, with $a\in\R^n$ and $T\in GL(n)$, the set of non-degenerate $n\times n$ linear maps, such that $a+TK$ is isotropic and then one can define the isotropic constant for every non-necessarily isotropic convex body as the isotropic constant of its isotropic images. It is also well known (see, for instance, \cite[Lemma 4.1]{MP}) that, among all the $n$-dimensional convex bodies, the Euclidean ball $B_2^n:=\{x\in\R^n\,:\Vert x\Vert_2\leq 2\}$, where $\Vert\cdot\Vert_2$ denotes the Euclidean norm in $\R^n$, is the one with the smallest isotropic constant and that its value is bounded below by a positive absolute constant, independent of $K$ and the dimension $n$. We call a constant absolute whenever it is independent of every other parameter involved.

However, it is still an open problem, known as the slicing problem, whether there exists an absolute constant bounding from above the isotropic constant of every isotropic convex body in every dimension. This question was posed by Bourgain, who proved in \cite{Bou} the following dimension-dependent upper bound for the isotropic constant of any convex body $K\subseteq\R^n$ : $L_K\leq Cn^{1/4}\log n$, where $C$ is an absolute constant. This upper bound was improved by Klartag \cite{K1}, removing the logarithm in Bourgain's upper bound. It was only recently when a major breakthough occurred and an upper bound of an order smaller that any power of $n$ has been obtained. Such a result is deduced from Chen's estimate of the best constant in the Kannan-Lov\'asz-Simonovits spectral gap conjecture \cite{Ch} and the known relation between both conjectures (see \cite{EK}). This result has been improved by Klartag and Lehec \cite{KL}, and Jambulapati, Lee and Vempala \cite{JLV}, which provided polylogarithmic in the dimension upper bounds for the isotropic constant of convex bodies.

This conjecture has been verified for several classes of convex bodies such us $1$-unconditional convex bodies, (see \cite{Bou0} or \cite{MP}), duals to zonoids (see \cite{B2} or \cite{EM}), bodies with a bounded outer volume ratio \cite{MP}, random bodies \cite{KK}, unit balls of Schatten norms \cite{KMP}, or polytopes with a small number of vertices \cite{ABBW}, and several reductions of the problem have been proved like a reduction to convex bodies with small diameter \cite{Bou} (see also \cite[Proposition 3.3.3]{BGVV}), or a reduction to convex bodies with finite volume ratio \cite{BKM}.

In this paper we are going to focus on the reduction to centrally symmetric convex bodies, which was proved by Klartag in \cite[Lemma 2.3]{K0} (see also \cite[Proposition 2.5.10]{BGVV}). A particular case of this results states the following:

Given an isotropic convex body $K\subseteq\R^n$, there exists another isotropic centrally symmetric convex body $T\subseteq\R^n$ and an absolute constant $C>0$ such that
\begin{equation}\label{eq:Klartag}
L_K\leq CL_T.
\end{equation}
Therefore, if the hyperplane conjecture is verified by centrally symmetric convex bodies, it is verified by every convex body. When examining the proof of this result, one observes the following three things:
\begin{enumerate}
\item[1)] This result is proved in the more general setting of log-concave functions (i.e, functions whose logarithm is concave), which is the context in which the hyperplane conjecture and other important conjectures are studied, stating that there exists an absolute constant $C>0$ such that if $f:\R^n\to[0,\infty)$ is an integrable isotropic log-concave function (i.e., with $\displaystyle{\int_{\R^n}f(x)dx=1}$, $\displaystyle{\int_{\R^n}xf(x)dx=0}$, and $\displaystyle{\int_{\R^n}\langle x,\theta\rangle^2f(x)dx=1}$ for every $\theta\in S^{n-1}$) there exists a centrally symmetric convex body $T\subseteq\R^n$ such that
\begin{equation}\label{eq:Klartag2}
L_f:=\left(\sup_{x\in\R^n}f(x)\right)^\frac{1}{n}\leq CL_T.
\end{equation}
\item[2)] The convex body $T$ appears as the convex body $K_{n+2}(g_K)$, where $g_K$ denotes the covariogram function of $K$ and, for a given log-concave function $g$, $(K_p(g))_{p>0}$ denotes a uniparametric family of convex bodies associated to $g$, which was introduced by Ball in \cite{B}, and which is defined, for $g:\R^n\to[0,\infty)$ a log-concave function with $g(0)>0$, and $p>0$, as
    $$
    K_p(g):=\left\{x\in\R^n\,:\,\int_0^\infty pt^{p-1}g(tx)dt\geq g(0)\right\}.
    $$
\item[3)] The constant $C$ appears as an upper bound of the ratio of the volume of the convex body $K_{n+2}(g_K)$ and the volume of $K_{n}(g_K)$, which is obtained by a general inclusion relation between the convex bodies $(K_p(g))_{p>0}$ for a general log-concave function $g$.
\end{enumerate}

In this paper we obtain the value of the best possible constant in relation \eqref{eq:Klartag} when one considers such relation in the setting of convex bodies instead of the absolute constant that appears in the more general setting of log-concave functions stated in equation \eqref{eq:Klartag2} and in the inclusion relation mentioned in 3). More precisely, we prove the following
\begin{theorem}\label{thm:TeoremaPrincipal}
Let, for every $n\in\N$, $D_n=\frac{1}{\sqrt{2}}\frac{{{2n}\choose{n}}^{\frac{1}{2}+\frac{1}{n}}}{{{2n+2}\choose{n}}^{\frac{1}{2}}}$.
Then, for every isotropic convex body $K\subseteq\R^n$
$$
L_K\leq D_nL_{K_{n+2}(g_K)}.
$$
Furthermore, there is equality if and only if $K$ is a regular simplex. Moreover, $\sup_{n\in\N}D_n=\sqrt{2}$.
\end{theorem}

In order to prove this theorem, we obtain a sharp inclusion relation between the convex bodies in the family $(K_p(g))_{p>0}$, whenever $g$ is an $\alpha$-concave function for some $\alpha>0$ (i.e., $g^\alpha$ is concave on its support), instead of log-concave, as when $K\subseteq\R^n$ is a convex body, the covariogram function $g_K$ is $\frac{1}{n}$-concave rather than log-concave. Let us point out that, while the natural setting to study many problems in convexity is the setting of log-concave functions, many sharp well-known inequalities in convex geometry appear as a consequence of some functions satisfying a better type of concavity (recall that any $\alpha$-concave function with $\alpha>0$ is, in particular, log-concave). Let us mention, for instance, Rogers-Shephard inequality  or Zhang's inequality, which can be obtained from the $\frac{1}{n}$-concavity of the covariogram function (see \cite{RS}, \cite{Z}, \cite{GZ}, or \cite{AJV}). Therefore, we believe that having sharp inclusion relations between the convex bodies in this family, which have typically been considered in the log-concave setting, for $\alpha$-concave functions might be interesting in itself. We state the relation in the following theorem (see notation below):

\begin{theorem} \label{thm:inlcusioninversain}
Let $g: \mathbb{R}^n \to [0, \infty ) $ be a continuous integrable $\alpha$-concave function such that $g(0) >0$. For any $0<p<q$, we have that
$$
 { \frac{1}{\alpha} + q \choose \frac{1}{\alpha}}^{1/q} K_q (g) \subseteq { \frac{1}{\alpha} + p \choose \frac{1}{\alpha}}^{1/p} K_p (g).
$$
Besides, if there exist $0<p<q$ such that this inclusion is an equality, then $\Vert g\Vert_\infty=g(0)$ and
$$
\frac{g(x)}{g(0)}=(1-\Vert x\Vert_L)^{1/\alpha}\chi_L(x),
$$
where $L=\textrm{supp}g$.
\end{theorem}

We use the following notation: For any convex body $K$ containing the origin in its interior, $\chi_K$ will denote its characteristic function, $\Vert\cdot\Vert_K$ will denote its Minkowski functional, defined as $\Vert x\Vert_K=\inf\{\lambda>0\,:\,x\in\lambda K\}$ for every $x\in\R^n$ and $\rho_K$ will denote its radial function, which is defined for every $u\in S^{n-1}$, as $\rho_K(u):=\max\{\lambda>0\,:\,\lambda u\in K\}$ and is extended homogeneously to $\R^n$.  Let us recall that $\rho_K(x)=\frac{1}{\Vert x\Vert_K}$ for every $x\in\R^n\setminus\{0\}$ and that for any two convex bodies $K,L$ containing the origin in their interiors and any $\lambda>0$, $\rho_{\lambda K}(x)=\lambda\rho_K(x)$ for every $x\in\R^n$ and $K\subseteq L$ if and only if $\rho_K(x)\leq\rho_L(x)$ for every $x\in\R^n$.  $d\sigma$ will denote the uniform probability measure on $S^{n-1}$ and, given a function $g:\R^n\to\R$, $\textrm{supp}g$ will denote the support of $g$ and $\Vert g\Vert_\infty$ will stand for the $\ell_\infty$-norm, or the supremum norm. For any $x,y>0$, we recall the definition of generalized binomial coefficients in terms of Gamma functions as

$$ {x\choose y} := \frac{\Gamma{(1+x)}}{\Gamma{(1+y)}\Gamma{(1+x-y)}}. $$

The paper is structured as follows. In Section \ref{sec:Preliminaries} we will introduce some preliminary known results that we will use in order to prove Theorem \ref{thm:TeoremaPrincipal} and Theorem \ref{thm:inlcusioninversain}. In Section \ref{sec:InclusionRelations} we will prove the inclusion relations given by Theorem \ref{thm:inlcusioninversain}. Finally, in Section \ref{sec:MainResult}, we will prove Theorem \ref{thm:TeoremaPrincipal}.

\section{Preliminaries}\label{sec:Preliminaries}

In this section we will introduce the necessary concepts and already known results that we need in our proofs.

\subsection{Covariogram function}

Given a convex body $K \subset \R^n$, its covariogram function is the function defined as
$$
g_K(x) = \left|K \cap (x + K)\right| = \int _{\mathbb{R}^n} \chi_{K}(y) \chi _{x+K} (y) dy.
$$
It is easy to check that $\Vert g_K\Vert_\infty=g_K(0)=|K|$ , that the support of $g_K$ is the difference body
$$
K-K=\{x-y\in\R^n\,:\,x-y\in K\},
$$
and that
$$
\int _{\mathbb{R}^n} g_K(x) dx = \int _{\mathbb{R}^n} \int _{\mathbb{R}^n} \chi_{K}(y) \chi _{x+K} (y) dy dx=\int _{\mathbb{R}^n} \int _{\mathbb{R}^n} \chi_{K}(y) \chi _{y-K} (x) dx dy = |K|^2.
$$
Therefore, if $|K|=1$ then $g_K$ is a probability density. Besides, for any convex body  $K\subseteq\R^n$, $g_K$ is an even function: For every $x\in\R^n$ we have
$$
g_K(-x) = |K \cap (-x + K) | = |(x+K) \cap (x -x + K) | = |(x+K) \cap K | = g_K(x).
$$
Moreover, as a consequence of Brunn-Minkowski inequality (see, for instance, \cite[Theorem 7.1.1]{Sch}) $g_K$ is a $\frac{1}{n}$-concave function.  We will make use of the following lemma, which is a simplified version of \cite[Proposition 2.10]{AJV}:
\begin{lemma}\label{lem:simplexcaracterization}
Let $K\subseteq\R^n$ be a convex body. Then $K$ is an $n$-dimensional simplex if and only if for every $\theta\in [0,1]$
$$
(1-\theta^{1/n})(K-K)=\{x\in K-K\,:\, g_K(x)\geq\theta|K|\}.
$$
\end{lemma}
The following lemma is also well known (see \cite[Proposition 2.5.10]{BGVV}). Nevertheless, since it is one of the key steps in the proof of inequality \eqref{eq:Klartag} and in the proof of Theorem \ref{thm:TeoremaPrincipal}, we will give its proof for the sake of completeness.
\begin{lemma} \label{lem:covariogramInertia}
Let $K \subset \mathbb{R}^n$ be a centered convex body with $|K|=1$. Let $g_K(x) = |K \cap (x + K)|$ be the covariogram function. Then, $\forall \theta \in S^{n-1}$,
$$
\int _{\mathbb{R}^n} \langle x , \theta \rangle ^2 g_K(x) dx =
2 \int _K \langle x , \theta \rangle ^2 dx.
$$
\end{lemma}

\begin{proof}

From the definition of $g_K$ we have
\begin{align*}
\int _{\mathbb{R}^n} \langle x , \theta \rangle ^2 g_K(x) dx & =
\int _{\mathbb{R}^n} \int _{\mathbb{R}^n} \langle x , \theta \rangle ^2  \chi_{K}(y) \chi _{x+K} (y) dx dy.
\end{align*}
Note that $y \in x + K \Leftrightarrow x \in y - K$. Also, using Fubini's theorem, we have that
\begin{align*}
&\int _{\mathbb{R}^n} \langle x , \theta \rangle ^2 g_K(x) dx  =
\int _{\mathbb{R}^n} \int _{\mathbb{R}^n} \langle x , \theta \rangle ^2 \chi_{K}(y) \chi _{y-K} (x) dx dy
\\
 & = \int _{\mathbb{R}^n} \int _{\mathbb{R}^n} \langle y + (x- y) , \theta \rangle ^2
 \chi_{K}(y) \chi _{y-K} (x) dx dy
\\
 & = \int _{\mathbb{R}^n} \int _{\mathbb{R}^n} \left( \langle y  , \theta \rangle ^2
 + 2 \langle y  , \theta \rangle \langle x- y  , \theta \rangle
 + \langle x-y  , \theta \rangle ^2 \right) \chi_{K}(y) \chi _{y-K} (x) dx dy.
\end{align*}
Splitting the latter integral in three integrals, and taking into account that $|K|=1$, we obtain
\begin{align*}
\int _{\mathbb{R}^n} \langle x , \theta \rangle ^2 g_K(x) dx & =  \int _{\mathbb{R}^n}  \langle y  , \theta \rangle ^2 \chi_{K}(y) dy |K|
\\
& + 2 \int _{\mathbb{R}^n} \int _{\mathbb{R}^n} \langle y  , \theta \rangle \langle x- y  , \theta \rangle
 \chi_{K}(y) \chi _{y-K} (x) dx dy
\\
 & + \int _{\mathbb{R}^n} \int _{\mathbb{R}^n}
 \langle x-y  , \theta \rangle ^2  \chi_{K}(y) \chi _{y-K} (x) dx dy.
\end{align*}
Finally, with the change of variables $x-y=z$ and taking into account that $|K|=1$ and that $K$ is centered, we have that
\begin{align*}
\int _{\mathbb{R}^n} \langle x , \theta \rangle ^2 g_K(x) dx & =  \int _K  \langle y  , \theta \rangle ^2  dy
 + 2 \int _{K} \int _{-K} \langle y  , \theta \rangle \langle z  , \theta \rangle  dz dy
+ \int _{K} \int _{-K} \langle z  , \theta \rangle ^2 dz dy
\\
 & =  \int _K  \langle y  , \theta \rangle ^2  dy + 0 +  \int _K  \langle z  , \theta \rangle ^2  dz = 2  \int _K  \langle x  , \theta \rangle ^2  dx.
\end{align*}
\end{proof}

\subsection{Ball's bodies}

Given a log-concave function  $g:\R^n\to[0,\infty)$ with $g(0)>0$ and $p>0$, we consider the following convex body, introduced by Ball in \cite{B}, where he also proved their convexity:
    $$
    K_p(g):=\left\{x\in\R^n\,:\,\int_0^\infty pt^{p-1}g(tx)dt\geq g(0)\right\}.
    $$
It is not difficult to see that their radial function is defined as
$$
\rho _{K_p(g)} (u) = \left( \frac{1}{g(0)} \int _{0} ^\infty p t^{p-1} g(tu)dt \right) ^{1/p},\quad u\in S^{n-1}.
$$
This family of convex bodies play an important role in the study of log-concave functions, as they allow us to construct convex bodies associated to log-concave functions preserving some of its properties. One can verify that if $g$ is an even function then, for every $p>0$ and every $u\in S^{n-1}$ $\rho _{K_p (g)} (-u)=\rho _{K_p (g)} (u)$. Therefore, if $g$ is an even log-concave function $K_p(g)$ is a centrally symmetric convex body for every $p>0$. The following lemma is also well known (see \cite[Proposition 2.5.3]{BGVV}). We include the proof here, as it provides another key step in the proof of \eqref{eq:Klartag} and Theorem \ref{thm:TeoremaPrincipal}.

\begin{lemma} \label{lem:IntegrationOnBallsBodies}
Let $g:\R^n \to [0, \infty)$ be a measurable function. Then, for every $\theta \in S^{n-1}$ and $p\geq0$ we have that
$$
\int _{K_{n+p}(g)} \left| \langle x , \theta \rangle \right|^p dx =
 \frac{1}{g(0)} \int _{\mathbb{R}^n} \left| \langle x , \theta \rangle \right|^p g(x) dx.
$$
\end{lemma}

\begin{proof}
By integrating in polar coordinates, we have that for any $p\geq0$ and any $\theta \in S^{n-1}$,
\begin{align*}
\int _{K_{n+p}(g)} \left| \langle x , \theta \rangle \right|^p dx & =
n |B^n _2|\int_{S^{n-1}} \int _0 ^{\rho _{K_{n+p}(g)}(u)}  \left| \langle ru , \theta \rangle \right|^p r^{n-1} dr
d \sigma (u)
\\
 & =  n |B^n _2|\int_{S^{n-1}} \left| \langle u , \theta \rangle \right|^p   \int _0 ^{\rho _{K_{n+p}(g)}(u)}  r^{n+p-1} dr d \sigma (u)
\\
 & = n |B^n _2| \int_{S^{n-1}} \left| \langle u , \theta \rangle \right|^p  \frac{\left(\rho _{K_{n+p}(g)}(u)\right)^{n+p}}{n+p}
   d \sigma (u).
\end{align*}
Using the expression of the radial function of $K_{n+p} (g)$ and integration in polar coordinates again we obtain
\begin{align*}
\int _{K_{n+p}(g)} \left| \langle x , \theta \rangle \right|^p dx & =
 n |B^n _2|\int_{S^{n-1}} \left| \langle u , \theta \rangle \right|^p
\frac{1}{g(0)} \int _0 ^{\infty} s^{n+p-1} g(su) ds d \sigma (u)
\\
 & = n \frac{|B^n _2|}{g(0)} \int_{S^{n-1}} \int _0 ^{\infty} s^{n-1}
 \left| \langle su , \theta \rangle \right|^p g(su) ds   d \sigma (u)
\\
& = \frac{1}{g(0)} \int _{\mathbb{R}^n} \left| \langle x , \theta \rangle \right|^p g(x) dx .
\end{align*}
\end{proof}

Finally, the following inclusion relation between different convex bodies in the family $(K_p(g))_{p>0}$ is also known (see \cite[Proposition 2.5.7]{BGVV}).  If $g: \mathbb{R}^n \to [0, \infty )$ is a log-concave function with $g(0) = \Vert g \Vert _{\infty}$ and $0<p<q$, then
\begin{equation}\label{eq:ContenidosBall}
\frac{\Gamma (1+p)^{1/p}}{\Gamma (1+q)^{1/q}} K_q (g) \subseteq  K_p (g)  \subseteq K_q(g).
\end{equation}

Let us point out that the first inclusion, which is the one used in the proof of \eqref{eq:Klartag}, does not need the assumption $g(0) = \Vert g \Vert _{\infty}$. Theorem \ref{thm:inlcusioninversain}, which will be proved in the next section, will improve the first inclusion relation whenever $g$ is $\alpha$-concave for some $\alpha>0$, leading to the obtention of the constant $D_n$ in Theorem \ref{thm:TeoremaPrincipal}.

\section{Inclusion relations between Ball Bodies}\label{sec:InclusionRelations}

In this section we are going to prove Theorem \ref{thm:inlcusioninversain}. It will be a direct consequence of the following lemma:

\begin{lemma}
\label{riptus}
Let $g:[0, \infty ) \to [0, \infty)$ be an integrable $\alpha$-concave function, with $g(0)>0$. Then, the function
\begin{align*}
G_g(p) = \left(  { \frac{1}{\alpha} + p \choose \frac{1}{\alpha}} \frac{1}{g(0)}  \int _0 ^{\infty} p t^{p-1} g(t) dt \right)^{1/p}
\end{align*}
is decreasing in $p\in(0,\infty)$. Furthermore, if there exist $0<p<q$ such that $G_g(p)=G_g(q)$ then there exists $M>0$ such that
\begin{gather*} \frac{g(t)}{g(0)} \; = \;
\begin{cases}
\left( 1 - \frac{t}{M} \right)^{1/\alpha} \; , \; \; \text{if }  t \in [0,M] \\
 \\
0 \; , \; \; \text{if } t \in (M, + \infty )
\end{cases}
\end{gather*}
\end{lemma}

\begin{rmk}
The proof of this lemma follows the idea of the proof of Berwald's inequality (see \cite{Be}, and also \cite{AAGJV} for an English translation). However, in the proof of Berwald's inequality, only the $\frac{1}{n}$-concave function given by the super-level sets of a concave function was considered. In this lemma we consider any 1-dimensional $\alpha$-concave function with $\alpha>0$. This will allow us to prove the aforementioned inclusion relation between Ball bodies.
\end{rmk}
\begin{proof}

Without loss of generality, we can assume that $g(0) = 1$. If $g(0)$ is not $1$, we apply the same argument with $g_1 = \frac{g}{g(0)}$, which satisfies that $g_1(0) = 1$ and is also integrable and $\alpha$-concave.

By definition, as $g$ is an $\alpha$-concave function, it follows that $g^{\alpha}$ is a concave function on its support, which is the same as the support of $g$. Besides, since $g$ is integrable and $\alpha$-concave, it has compact support. We define
\begin{align*}
M =\sup  \left\lbrace t \in [0, \infty ) \; ; \; g(t) \neq 0 \right\rbrace .
\end{align*}
Then, it is clear that for every $p>0$
\begin{align*}
G_g(p) = \left(  { \frac{1}{\alpha} + p \choose \frac{1}{\alpha}}  \int _0 ^{\infty} p t^{p-1} g(t) dt \right)^{1/p}
= \left(  { \frac{1}{\alpha} + p \choose \frac{1}{\alpha}}  \int _0 ^{M} p t^{p-1} g(t) dt \right)^{1/p} .
\end{align*}

Given a constant $M_1>0$, we define $h(t)$ as
\begin{gather*} h(t) \; = \;
\begin{cases}
\left( 1 - \frac{t}{M_1} \right)^{1/\alpha} \; , \; \; \text{if }  t \in [0,M_1] \\
 \\
0 \; , \; \; \text{if } t \in (M_1, + \infty )
\end{cases}
\end{gather*}
Note that $h^{\alpha}(t)$ is affine on its support, so it is also concave. Then, $h$ is an $\alpha$-concave function.
Changing variables  $t=M_1x$, it follows that
\begin{align*}
G_{h}(p) ^p & =  { \frac{1}{\alpha} + p \choose \frac{1}{\alpha}}  \int _0 ^{\infty} p t^{p-1} h(t) dt = { \frac{1}{\alpha} + p \choose \frac{1}{\alpha}}  \int _0 ^{M_1} p t^{p-1} \left( 1 - \frac{t}{M_1} \right)^{1/\alpha}  dt
\\
 & = { \frac{1}{\alpha} + p \choose \frac{1}{\alpha}}
 \frac{1}{M_1^{1 / \alpha}}  \int _0 ^{M_1} p t^{p-1} \left( M_1 - t \right)^{1/\alpha}  dt
\\
 & = { \frac{1}{\alpha} + p \choose \frac{1}{\alpha}}
 \frac{1}{M_1^{1 / \alpha}}  \int _0 ^{1} p (M_1x)^{p-1} \left( M_1 - M_1x \right)^{1/\alpha}  dx M_1
\\
 & = { \frac{1}{\alpha} + p \choose \frac{1}{\alpha}}
 M_1^p  p \int _0 ^{1}   x^{p-1} \left( 1 - x \right)^{1/\alpha}  dx
\\
 & = { \frac{1}{\alpha} + p \choose \frac{1}{\alpha}}
 M_1^p p
 \left(  \frac{\Gamma(p) \Gamma(1/ \alpha + 1) }{\Gamma(p + 1/ \alpha + 1)}   \right)
\\
 & = \left(  \frac{\Gamma(p + 1/ \alpha + 1)}{\Gamma(1/ \alpha + 1)
 \Gamma(p+1)}   \right) M_1^p  p
 \left(  \frac{\Gamma(p) \Gamma(1/ \alpha + 1) }{\Gamma(p + 1/ \alpha + 1)}   \right)
\\
 & = M_1^p .
\end{align*}
Then, it is clear that $G_h(p)=M_1$ for every $p>0$. Let us now take $0<p<q$ and fix
$$
M_1 = G_g(p).
$$
The equality $G_h(p) = M_1=G_g(p)$ means that
\begin{align*}
\left(  { \frac{1}{\alpha} + p \choose \frac{1}{\alpha}}  \int _0 ^{\infty} p t^{p-1} g(t) dt \right)^{1/p}
= \left(  { \frac{1}{\alpha} + p \choose \frac{1}{\alpha}}  \int _0 ^{\infty} p t^{p-1} h(t) dt \right)^{1/p} .
\end{align*}
Equivalently,
\begin{align}
\label{equivalentementeee}
\int _0 ^{\infty} p t^{p-1} g(t) dt
= \int _0 ^{\infty} p t^{p-1} h(t) dt.
\end{align}

Notice that necessarily $M\leq M_1$. Otherwise, if $M>M_1$, we would have that since $g^\alpha$ is concave on $[0,M]$ with $g(0)=1$, for every $t\in(0,M]$
$$
h^\alpha(t)<\left( 1 - \frac{t}{M} \right)\leq g^\alpha(t)
$$
and then
$$
M_1^p=G_g(p)^p={ \frac{1}{\alpha} + p \choose \frac{1}{\alpha}}  \int _0 ^{M} p t^{p-1} g(t) dt>{ \frac{1}{\alpha} + p \choose \frac{1}{\alpha}}  \int _0 ^{M_1} p t^{p-1} h(t) dt=M_1^p,
$$
which is a contradiction. Moreover, since the two integrals in \eqref{equivalentementeee} are equal and $g(0)=h(0)=1$, it cannot happen that $ h(t) < g(t)$ for every $t \in (0, +\infty)$. Thus, there exists some $t \in (0, +\infty)$ for which $h(t) \geq g(t)$. Then, the set $\lbrace t>0 : h(t) \geq g(t) \rbrace$ is non-empty. We take the infimum of this set
$$
t_0 = \inf \left\lbrace t>0 : \left( 1- \frac{t}{M_1} \right)^{1/\alpha} \geq g(t) \right\rbrace.
$$
By definition of $t_0$, it is clear that
$$
h(t) < g(t), \; \; \forall t \in (0, t_0).
$$
Let us see that $h(t) \geq g(t)$ for every $t > t_0$. On the one hand, it is clear that for every $t>M$ we have that $h(t) \geq g(t) = 0$. On the other hand, if $t_0<M$ then for any $t_0<y\leq M$, we take $\lambda >0$ such that $t_0 = \lambda 0 + (1 - \lambda ) y$. Then, as $g^\alpha$ is concave on $[0,M]$ and $h^\alpha$ is affine on $[0,M_1]$, we have that
\begin{align*}
 & g^\alpha(t_0) =g^\alpha(\lambda 0 + (1 - \lambda ) y) \geq \lambda g^\alpha(0) + (1- \lambda)g^\alpha(y)
\\
 & h^\alpha(t_0) = h^\alpha(\lambda 0 + (1 - \lambda ) y) = \lambda h^\alpha(0) + (1- \lambda)h^\alpha(y).
\end{align*}
From the definition of $t_0$ and the continuity of $g$ and $h$, since concave functions are continuous on the interior of their supports, $g(t_0) = h(t_0)$. Then, it follows that
\begin{align*}
\lambda h^\alpha(0) + (1- \lambda)h^\alpha(y) \geq \lambda g^\alpha(0) + (1- \lambda)g^\alpha(y).
\end{align*}
As $g(0) = h(0) = 1$,
\begin{align*}
h(y) \geq g(y).
\end{align*}
Then, for every $y \in (t_0 , M)$ we have that $h(y) \geq g(y)$.

Rewriting equality \eqref{equivalentementeee}, we have that
\begin{align*}
\frac{1}{{ \frac{1}{\alpha} + p \choose \frac{1}{\alpha}}} \left( G_g(p)^p - G_h(p)^p \right) = \int _0 ^{\infty}p t^{p-1} (g(t) - h(t) ) dt = 0.
\end{align*}
Thus, it follows that
\begin{align}
\label{berwaldacot}
\int _0 ^{t_0}p t^{p-1} (g(t) - h(t) ) dt -
\int _{t_0} ^{\infty}p t^{p-1} (h(t) - g(t) ) dt = 0.
\end{align}
Then, for every $q >p$ we have that
\begin{align*}
\frac{G_g(q)^q - G_h(q)^q}{{ \frac{1}{\alpha} + q \choose \frac{1}{\alpha}}}
 &  = \int _0 ^{\infty}q t^{q-1} (g(t) - h(t) ) dt
\\
 & = \int _0 ^{t_0}q t^{q-1} (g(t) - h(t) ) dt
  - \int _{t_0} ^{\infty} q t^{q-1} (h(t) - g(t) ) dt
\\
 & = \frac{q}{p} \left(
 \int _0 ^{t_0} p t^{p-1} t^{q-p} (g(t) - h(t) ) dt
 - \int _{t_0} ^{\infty} p t^{p-1} t^{q-p} (h(t) - g(t) ) dt \right)
\\
 & \leq \frac{q}{p} t_0^{q-p} \left(
 \int _0 ^{t_0} p t^{p-1} (g(t) - h(t) ) dt
 - \int _{t_0} ^{\infty} p t^{p-1} (h(t) - g(t) ) dt \right).
\end{align*}
Using equality \ref{berwaldacot}, it is clear that
\begin{align*}
\frac{G_g(q)^q - G_h(q)^q}{{ \frac{1}{\alpha} + q \choose \frac{1}{\alpha}}}
& \leq \frac{q}{p} t_0^{q-p} \left(
 \int _0 ^{t_0} p t^{p-1} (g(t) - h(t) ) dt
 - \int _{t_0} ^{\infty} p t^{p-1} (h(t) - g(t) ) dt \right)
\\
 & = \frac{q}{p} t_0^{q-p} \cdot 0 = 0.
\end{align*}
So, we have that for every $0<p<q$,
\begin{align*}
G_g(q) - G_h(q) \leq 0.
\end{align*}
From the definition of $h$, it satisfies that $G_h(q) = G_h(p) = G_g(p)$. So, we can conclude that for every $0<p<q$,
\begin{align*}
G_g(q) - G_g(p) \leq 0.
\end{align*}
and its proved that the function $G_g(p)$ is decreasing in $p \in (0, \infty )$.

Besides, if there exist $0<p<q$ such that $G_g(p)=G_g(q)$, then, defining $h$ as before, we have that necessarily $g(t)=h(t)$ for every $t>0$.
\end{proof}

We can obtain now the proof of Theorem \ref{thm:inlcusioninversain}.

\begin{proof}[Proof of Theorem \ref{thm:inlcusioninversain}]
Let $g:\R^n\to[0,\infty)$ be a continuous integrable $\alpha$-concave function with $g(0)>0$ and let $u\in S^{n-1}$. Calling $g_u:[0,\infty)\to[0,\infty)$ the integrable $\alpha$-concave function $g_u(t)=g(tu)$ we obtain that for every $p>0$
$$
{ \frac{1}{\alpha} + p \choose \frac{1}{\alpha}}^{1/p}
\rho _{K_p(g)} (u)  = \left( { \frac{1}{\alpha} + p \choose \frac{1}{\alpha}} \frac{1}{g(0)} \int _0 ^{\infty} p t^{p-1} g(tu) dt \right)^{1/p}
= G_{g_u}(p),
$$
where $G_{g_u}(p)$ is defined, as in the previous lemma, as
$$
G_{g_u}(p) = \left(  { \frac{1}{\alpha} + p \choose \frac{1}{\alpha}} \frac{1}{g_u(0)}  \int _0 ^{\infty} p t^{p-1} g_u(t) dt \right)^{1/p}.
$$
Since, by Lemma \ref{riptus}, $G_{g_u}(p)$ is decreasing in $p\in(0,\infty)$ we have that for every $0<p<q$ and for every $u\in S^{n-1}$
$$
{\frac{1}{\alpha} + q \choose \frac{1}{\alpha}}^{1/q}
\rho _{K_q(g)} (u)=G_{g_u}(q) \leq G_{g_u}(p)={ \frac{1}{\alpha} + p \choose \frac{1}{\alpha}}^{1/p}
\rho _{K_p(g)} (u).
$$
Therefore,
$$
 { \frac{1}{\alpha} + q \choose \frac{1}{\alpha}}^{1/q} K_q (g) \subseteq { \frac{1}{\alpha} + p \choose \frac{1}{\alpha}}^{1/p} K_p (g),
$$
which proves the inclusion. Assume now that there exist $0<p<q$ such that the above inclusion is an equality. Then, for every $u\in S^{n-1}$ we have that
$$
G_{g_u}(q) = G_{g_u}(p)
$$
and, by the equality case in Lemma \ref{riptus}, we have that for every $u\in S^{n-1}$ there exists $M_u>0$ such that
\begin{gather*} \frac{g_u(t)}{g(0)} \; = \;  \frac{g_u(t)}{g_u(0)} \; = \;
\begin{cases}
\left( 1 - \frac{t}{M_u} \right)^{1/\alpha} \; , \; \; \text{if }  t \in [0,M_u] \\
 \\
0 \; , \; \; \text{if } t \in (M_u, + \infty )
\end{cases}
\end{gather*}
Then, necessarily, if $L=\textrm{supp}g$, we have that $M_u=\rho_L(u)=\frac{1}{\Vert u\Vert_L}$ and then for every $u\in S^{n-1}$ and every $t\in[0,\infty)$
\begin{gather*} \frac{g(tu)}{g(0)} \; = \;
\begin{cases}
\left( 1 - \Vert tu\Vert_L \right)^{1/\alpha} \; , \; \; \text{if }  tu \in L \\
 \\
0 \; , \; \; \text{if } tu \not\in L
\end{cases}
\end{gather*}
Equivalently, $\frac{g(x)}{g(0)}=(1-\Vert x\Vert_L)^{1/\alpha}\chi_L(x)$ for every $x\in\R^n$.
\end{proof}

\section{Proof of the main result}\label{sec:MainResult}

In this section we are going to prove Theorem \ref{thm:TeoremaPrincipal}. For that matter, we take $K\subseteq\R^n$ an isotropic convex body and we proceed as in the proof of inequality \eqref{eq:Klartag}: We consider the Ball body $K_{n+2}(g_K)$, where $g_K$ denotes the covariogram function, and observe that $T:=|K_{n+2}(g_K)|^{-1/n}K_{n+2}(g_K)$ has volume 1, it is centrally symmetric, as $g_K$ is an even function, and by Lemma \ref{lem:IntegrationOnBallsBodies} with $p=2$ and Lemma \ref{lem:covariogramInertia}, for every $\theta\in S^{n-1}$
\begin{eqnarray*}
\int_T\langle x,\theta\rangle^2dx&=&\frac{1}{|K_{n+2}(g_K)|^{1+2/n}}\int_{K_{n+2}(g_K)}\langle x,\theta\rangle^2dx=\frac{1}{|K_{n+2}(g_K)|^{1+2/n}}\int_{\R^n}\langle x,\theta\rangle^2g_K(x)dx\cr
&=&\frac{2}{|K_{n+2}(g_K)|^{1+2/n}}\int_{K}\langle x,\theta\rangle^2dx=\frac{2L_K^2}{|K_{n+2}(g_K)|^{1+2/n}}.
\end{eqnarray*}
Therefore, $T$ is isotropic with
\begin{equation}\label{eq:identityIsotropicConstants}
L^2_{K_{n+2}(g_K)}=L_T^2=\frac{2L_K^2}{|K_{n+2}(g_K)|^{1+2/n}}
\end{equation}
Observing now that $g_K$ is a $\frac{1}{n}$-concave function, we have, by Theorem \ref{thm:inlcusioninversain}, that
\begin{equation}\label{eq:InclusionEnPrueba}
{{2n+2}\choose{n}}^{1/(n+2)}K_{n+2}(g_K)\subseteq{{2n}\choose{n}}^{1/n}K_n(g_K)
\end{equation}
and, taking volumes and taking into account that by Lemma \ref{lem:IntegrationOnBallsBodies} with $p=0$ we have that $|K_n(g_K)|=\int_{\R^n}g_K(x)dx=1$ we obtain that
$$
|K_{n+2}(g_K)|\leq\frac{{{2n}\choose{n}}}{{{2n+2}\choose{n}}^{n/(n+2)}},
$$
which, together with \eqref{eq:identityIsotropicConstants}, yields
$$
L^2_{K_{n+2}}(g_K)\geq 2\frac{{{2n+2}\choose{n}}}{{{2n}\choose{n}}^{1+2/n}}=\frac{L_K^2}{D_n^2},
$$
which is the inequality in Theorem \ref{thm:TeoremaPrincipal}. If there is equality in the above inequality we have that \eqref{eq:InclusionEnPrueba} is an equality. Therefore, by the characterization of the equality in Theorem \ref{thm:inlcusioninversain}, we have that
$$
g_K(x)=(1-\Vert x\Vert_{K-K})^n\chi_{K-K}(x)\quad\forall x\in\R^n.
$$
Therefore, for every $\theta\in[0,1]$ we have that
$$
(1-\theta^{1/n})(K-K)=\{x\in K-K\,:\, g_K(x)\geq\theta|K|\}.
$$
and, by Lemma \ref{lem:simplexcaracterization}, $K$ is an $n$-dimensional simplex. Since we are assuming that $K$ is isotropic, then $K$ is the $n$-dimensional regular simplex with volume 1.

Finally, let us prove that $\displaystyle{\sup_{n\in\N}D_n=\sqrt{2}}$ or, equivalently, that
\begin{equation}\label{eq:infimoAProbar}
\inf \left\lbrace  2 \frac{{2n + 2\choose n}}{ {2n\choose n}^{1+2/n} }\,:\, n \in \mathbb{N} \right\rbrace = \frac{1}{2}.
\end{equation}

Using Stirling's formula, one can check that 
$$
\lim _{n \to \infty} 2 \frac{{2n + 2\choose n}}{ {2n\choose n}^{1+2/n} }
= \lim _{n \to \infty} 2 \frac{\Gamma(2n+3)}{\Gamma(n+3)\Gamma(n+1)} \cdot
\left( \frac{\Gamma(n+1)\Gamma(n+1)}{\Gamma(2n+1)} \right)^{1+2/n}
 =  \frac{1}{2}.
$$
Therefore, in order to prove \eqref{eq:infimoAProbar}, it is enough to prove that for every $n\geq 1$
\begin{equation}\label{eq:ecuacionAProbar}
2 \frac{{2n + 2\choose n}}{ {2n\choose n}^{1+2/n} }\geq \frac{1}{2}.
\end{equation}

We will prove \eqref{eq:ecuacionAProbar} by induction. In order to prove the induction step, we will split the proof in two lemmas. In the first lemma we will prove an auxiliary inequality that we will need in order to prove the induction step. In the second lemma we will actually prove the induction step:

\begin{lemma}
\label{mayorque1}
For every $n \in \N$,
\begin{align*}
\frac{4n}{n+2} \left( \frac{n (n+1)}{(n-1)(n+2)} \right)^{n-1} \frac{n^2}{(2n-1)^2} \geq 1
\end{align*}
\end{lemma}

\begin{proof}

Notice that for every $n\in\N$
\begin{align*}
\frac{4n}{n+2} \left( \frac{n (n+1)}{(n-1)(n+2)} \right)^{n-1} \frac{n^2}{(2n-1)^2} & \geq 1
\\
 \Leftrightarrow \left( \frac{n (n+1)}{(n-1)(n+2)} \right)^{n-1} & \geq  \frac{(2n-1)^2}{n^2} \frac{(n+2)}{4n}
\\
 \Leftrightarrow \left(1+ \frac{2}{n^2+n-2} \right)^{n-1} & \geq 1+ \frac{4n^2-7n+2}{4n^3}
\\
 \Leftrightarrow (n-1) \log \left( 1+ \frac{2}{n^2+n-2} \right) & \geq  \log \left( 1+ \frac{4n^2-7n+2}{4n^3} \right).
\end{align*}
Since for every $x > -1$,
\begin{align*}
\log (1 + x) \geq \frac{x}{1+x},
\end{align*}
we obtain that
\begin{align*}
(n-1) \log \left( 1+ \frac{2}{n^2+n-2} \right) & \geq
(n-1) \frac{\frac{2}{n^2 + n-2}}{1+ \frac{2}{n^2 + n-2}}
 = \frac{2(n-1)}{n(n+1)}.
\end{align*}
Also, notice that
\begin{align*}
\frac{2(n-1)}{n(n+1)}  \geq \frac{4n^2 - 7n + 2}{4n^3}
&\Leftrightarrow 8n^3(n-1)  \geq n(n+1)(4n^2-7n+2)\\
&\Leftrightarrow 4n^4 -5n^3 + 5n^2 - 2n \geq 0\\
&\Leftrightarrow 4n^3 -5n^2 +5n -2 \geq 0.
\end{align*}
As $g(t)=4t^3 -5t^2 +5t -2$ is an increasing function on the interval $[0, \infty)$ and $g(1) = 2$, it is clear that $4n^3 -5n^2 +5n -2 \geq 0$ for every $n\geq 1$.
So, using that
\begin{align*}
x \geq \log (1 + x)
\end{align*}
for every $x>-1$, it follows that
\begin{align*}
(n-1) \log \left( 1+ \frac{2}{n^2+n-2} \right) & \geq
\frac{2(n-1)}{n(n+1)} \geq \frac{4n^2 - 7n + 2}{4n^3}\\
&\geq \log \left( 1+ \frac{4n^2-7n+2}{4n^3} \right).
\end{align*}
\end{proof}

Now let us prove that equation \eqref{eq:ecuacionAProbar} is true. In  order to slightly ease the notation, we will make use of Catalan numbers. These numbers have been widely studied in the literature (see, for instance \cite{catalan} or \cite{St} ). They are defined in the following way: for each $n \in \N$, the $n$-th Catalan number is defined as
$$
C_n = \frac{1}{n+1} {2n \choose n } .
$$
It is easy to check that for every $n \in \N$, Catalan numbers satisfy that
\begin{align*}
C_{n+1} = \frac{2(2n+1)}{n+2} C_n.
\end{align*}

\begin{lemma}\label{lem:Induction}
For every $n\geq 1$,
$$
2 \frac{{2n + 2\choose n}}{ {2n\choose n}^{1+2/n} } \geq \frac{1}{2}.
$$
\end{lemma}

\begin{proof}
Note that
\begin{align*}
{2n+2\choose n} & = \frac{(2n+2)!}{n! (n+2)!} =
 \frac{(2n+2)!(n+1)}{(n+1)!(n+1)! (n+2)} = \frac{1}{n+2} {2n+2\choose n+1} (n+1)
 = C_{n+1} (n+1).
\end{align*}
Then, using the properties of Catalan numbers, it follows that
\begin{align*}
2 \frac{{2n + 2\choose n}}{ {2n\choose n}^{1+2/n} }
= 2 \frac{C_{n+1}}{ \frac{1}{(n+1)} {2n\choose n} {2n\choose n}^{2/n} }
= 2 \frac{C_{n+1}}{C_n {2n \choose n}^{2/n}} =2  \frac{2 (2n+1)}{(n+2)} {2n \choose n}^{-2/n}.
\end{align*}
Thus, for every $n \geq 1$,
\begin{align*}
2 \frac{{2n + 2\choose n}}{ {2n\choose n}^{1+2/n} }  \geq \frac{1}{2} \; \;
 & \Leftrightarrow \;  2  \frac{2 (2n+1)}{(n+2)} {2n \choose n}^{-2/n}  \geq \frac{1}{2}
\\
 & \Leftrightarrow \; \frac{2n+1}{n+2}  \geq \frac{1}{8} {2n \choose n}^{2/n}
\\
 & \Leftrightarrow \; \left(  \frac{2n+1}{n+2} \right)^n  \geq
 \frac{1}{8^n} {2n \choose n}^{2}
\end{align*}
Using the fact that $\displaystyle{\left(  \frac{2n+1}{n+2} \right)^n \geq \left(  \frac{2n}{n+2} \right)^n}$ for every $n \geq 1$, and $\displaystyle{{2n \choose n} = (n+1) C_n}$, for proving this lemma it is enough to prove that
\begin{align}
\label{reduccfinal}
\left( \frac{16n}{n+2} \right)^n \geq (n+1)^2 C_n^2,
\end{align}
for every $n \geq 1$. We will prove \eqref{reduccfinal} by induction.

\smallskip

For $n=1$, using that $C_1 = 1$, the inequality is trivially satisfied:
\begin{align*}
\left( \frac{16 \cdot 1}{1+2} \right)^1 = \frac{16}{3} \geq 4 = (1+1)^2 C_1 .
\end{align*}

\smallskip

For $n\geq 1$, assume that \eqref{reduccfinal} is true for every $k\leq n-1$. Then,
\begin{align*}
\left( \frac{16n}{n+2} \right)^n & =
 \frac{16n}{n+2} \left( \frac{16n}{n+2} \right)^{n-1}
 = \frac{16n}{n+2} \left( \frac{n(n+1)}{(n-1)(n+2)} \right)^{n-1}
 \left( \frac{16(n-1)}{n+1} \right)^{n-1}
\\
 & \geq \frac{16n}{n+2} \left( \frac{n(n+1)}{(n-1)(n+2)} \right)^{n-1} n^2 C_{n-1}^2
\\
 & = \frac{16n}{n+2} \left( \frac{n(n+1)}{(n-1)(n+2)} \right)^{n-1}
  \frac{n^2}{(n+1)^2} \frac{C_{n-1}^2}{C_n^2} (n+1)^2 C_n^2 .
\end{align*}
Note that for every $n\in \N$,
$$
C_n = \frac{2(2n-1)}{n+1} C_{n-1}.
$$
Equivalently,
$$
\frac{C_{n-1}^2}{C_n^2} = \frac{(n+1)^2}{2^2 (2n-1)^2}.
$$
Then
\begin{align*}
\left( \frac{16n}{n+2} \right)^n & \geq \frac{16n}{n+2} \left( \frac{n(n+1)}{(n-1)(n+2)} \right)^{n-1}
  \frac{n^2}{(n+1)^2} \frac{C_{n-1}^2}{C_n^2} (n+1)^2 C_n^2
\\
 & = \frac{16n}{n+2} \left( \frac{n(n+1)}{(n-1)(n+2)} \right)^{n-1}
  \frac{n^2}{(n+1)^2} \frac{(n+1)^2}{2^2 (2n-1)^2} (n+1)^2 C_n^2
\\
 & = \frac{4n}{n+2} \left( \frac{n(n+1)}{(n-1)(n+2)} \right)^{n-1}
 \frac{n^2}{(2n-1)^2} (n+1)^2 C_n^2.
\end{align*}
By using Proposition \ref{mayorque1}, we have that
\begin{align*}
\left( \frac{16n}{n+2} \right)^n & \geq \frac{4n}{n+2} \left( \frac{n(n+1)}{(n-1)(n+2)} \right)^{n-1}
 \frac{n^2}{(2n-1)^2} (n+1)^2 C_n^2 \geq (n+1)^2 C_{n}^2 .
\end{align*}
\end{proof}







\noindent {\it Acknowledgements.}
I would like to thank my PhD. supervisor, David Alonso-Guti\'errez, for many helpful discussions and his valuable suggestions during the preparation of this paper.

\end{document}